\documentclass{article}
\usepackage[letterpaper]{geometry}

\usepackage{amsmath}
\usepackage{amsfonts}
\usepackage{amsthm}
\usepackage{amssymb}

\usepackage{color}

\usepackage{booktabs}            
\usepackage{diagbox}		 

\usepackage{algorithm}
\usepackage{algorithmicx}
\usepackage{algpseudocode}

\usepackage{eqparbox}

\usepackage[capitalize,nameinlink]{cleveref}
\usepackage{enumerate}
\usepackage{url}
\usepackage{mathtools}
\usepackage{graphicx}

\newtheorem{theorem}{Theorem}[section]
\newtheorem{corollary}[theorem]{Corollary}
\newtheorem{proposition}[theorem]{Proposition}
\newtheorem{lemma}[theorem]{Lemma}
\newtheorem{conjecture}[theorem]{Conjecture}
\newtheorem{definition}[theorem]{Definition}
\newtheorem{example}[theorem]{Example}
\newtheorem{question}[theorem]{Question}

\long\def\symbolfootnote[#1]#2{\begingroup%
\def\thefootnote{\fnsymbol{footnote}}\footnote[#1]{#2}\endgroup}

\allowdisplaybreaks[4]

\newcommand{\Ra}{\mathcal{R}}
\newcommand{\PSCA}{{\rm{PSCA}}}
\newcommand{\rv}{{\rm{rv}}}
\newcommand{\Cl}{{\rm{Cl}}}
\newcommand{\lsr}{\langle \sigma \rangle}

\newcommand{\PreserveBackslash}[1]{\let\temp=\\#1\let\\=\temp}
\newcolumntype{C}[1]{>{\PreserveBackslash\centering}p{#1}}
\newcolumntype{R}[1]{>{\PreserveBackslash\raggedleft}p{#1}}
\newcolumntype{L}[1]{>{\PreserveBackslash\raggedright}p{#1}}

\begin{document}

\title{A group-based structure for perfect sequence covering arrays}
\author{Jingzhou Na \and Jonathan Jedwab \and Shuxing Li}
\date{4 February 2022}
\maketitle

\symbolfootnote[0]{
Department of Mathematics,
Simon Fraser University, 8888 University Drive, Burnaby BC V5A 1S6, Canada.
\par
J.~Na is supported by a PhD Scholarship from the China Scholarship Council.
J.~Jedwab is supported by an NSERC Discovery Grant.
S.~Li is supported by a PIMS Postdoctoral Fellowship.
\par
Email:
{\tt jingzhou\_na@sfu.ca},
{\tt jed@sfu.ca},
{\tt shuxing\_li@sfu.ca}
}


\begin{abstract}
An $(n,k)$-perfect sequence covering array with multiplicity $\lambda$, denoted $\PSCA(n,k,\lambda)$, is a multiset whose elements are permutations of the sequence $(1,2, \dots, n)$ and which collectively contain each ordered length $k$ subsequence exactly $\lambda$ times.
The primary objective is to determine for each pair $(n,k)$ the smallest value of $\lambda$, denoted $g(n,k)$, for which a $\PSCA(n,k,\lambda)$ exists;
and more generally, the complete set of values~$\lambda$ for which a $\PSCA(n,k,\lambda)$ exists.
Yuster recently determined the first known value of $g(n,k)$ greater than~1, namely $g(5,3)=2$, and suggested that finding other such values would be challenging.
We show that $g(6,3)=g(7,3)=2$, using a recursive search method inspired by an old algorithm due to Mathon.
We then impose a group-based structure on a perfect sequence covering array by restricting it to be a union of distinct cosets of a prescribed nontrivial subgroup of the symmetric group~$S_n$. 
This allows us to determine the new results that $g(7,4)=2$ and $g(7,5) \in \{2,3,4\}$ and $g(8,3) \in \{2,3\}$ and $g(9,3) \in \{2,3,4\}$. 
We also show that, for each $(n,k) \in \{ (5,3), (6,3), (7,3), (7,4) \}$, there exists a $\PSCA(n,k,\lambda)$ if and only if $\lambda \ge 2$;
and that there exists a $\PSCA(8,3,\lambda)$ if and only if $\lambda \ge g(8,3)$.
\end{abstract}

\section{Introduction}
Throughout, let $k$ and $n$ be integers satisfying $2 \le k \le n$.
Let $S_n$ be the set of permutations of $[n] := \{1,2,\dots,n\}$,
and let $S_{n,k}$ be the set of $\frac{n!}{(n-k)!}$ ordered $k$-subsets of~$[n]$.
An \emph{$(n,k)$-perfect sequence covering array} with multiplicity $\lambda$, denoted \emph{$\PSCA(n,k, \lambda)$}, is a multiset $P$ with elements in $S_n$ such that each element of $S_{n,k}$ is a $k$-subsequence of exactly $\lambda$ elements of~$P$. Equivalently, regarding the elements of $P$ as $n$-sequences, we say that each element of $S_{n,k}$ is \emph{covered} by exactly $\lambda$ sequences of~$P$.
For example, the subset
\[
\{ 12345 ,\, 13254 ,\, 14523 ,\, 15432 ,\,
   24315 ,\, 25413 ,\, 34512 ,\, 35214 ,\,
   42513 ,\, 43215 ,\, 52314 ,\, 53412 
\}
\]
of $S_5$ is a $\PSCA(5,3,2)$.
If $P$ is a $\PSCA(n,k,1)$, then $P$ is a set (not a multiset). 
The size of the multiset of $k$-subsequences covered by a $\PSCA(n,k,\lambda)$ $P$ can be counted both as $|S_{n,k}|\lambda$ and as $\binom{n}{k}|P|$, from which we obtain the necessary condition $|P| = k! \lambda$.

Perfect sequence covering arrays are related to several other objects from combinatorial design theory.
A $\PSCA(n,k,\lambda)$ $P$ is equivalent to a $k$-$(n,n,\lambda)$ directed design~$([n],P)$ \cite[Sect.~VI.20]{BM07}. If a $\PSCA(n,k,\lambda)$ exists, then it achieves the largest possible size of a $k$-$(n,n,\lambda)$ directed packing \cite{MT99},~\cite{DSS84}. Replacing ``exactly $\lambda$ elements'' in the definition of a perfect sequence covering array by ``at least one element'' gives an $(n,k)$ sequence covering array, or equivalently a completely $k$-scrambling set of~$[n]$ \cite{Spe71}, \cite[Section~5]{Fur96}, \cite{Ish96}.
For a comprehensive study of constructions, nonexistence results, and search methods for sequence covering arrays, see Chee et al.\ \cite{CCHZ13}.
Sequence covering arrays are useful in various applications in which faults can arise when certain events occur in a particular order \cite{WSLK08,WLS+09,ARA09,HCM10,YM10,YCM11,KHL+12}. 
For example, the faults might be adverse reactions when a sequence of medications is taken in a certain order.
In order to determine whether faults arise under all possible ordered subsets of at most $k$ out of $n$ events, we require a set of tests in which each ordering of each subset of $k$ events occurs: this is given by an $(n,k)$-sequence covering array. 
A $\PSCA(n,k,1)$, if it exists, is the smallest possible size of an $(n,k)$-sequence covering array and so represents the most cost-efficient method of carrying out the required set of tests.

We define $g(n,k)$ to be the smallest $\lambda$ for which a $\PSCA(n,k,\lambda)$ exists. This value is well-defined, because $S_n$ is a trivial $\PSCA(n, k, \frac{n!}{k!})$ and so $g(n,k) \le \frac{n!}{k!}$ for all $k \le n$.
The central objective in the study of perfect sequence covering arrays is to determine, for each pair $(n,k)$, the value of $g(n,k)$ and more generally the complete set of values $\lambda$ for which there exists a $\PSCA(n,k,\lambda)$.
The current state of knowledge for $g(n,k)$ for small values of $n$ and $k$ is shown in \cref{tab:gvalues}.
We are concerned in this paper with exact values rather than asympotic bounds.

The rest of the paper is organized in the following way. 
\cref{sec:prev} describes previous results for the value of $g(n,k)$, including constructions, combinatorial nonexistence results, computer nonexistence results, and asymptotic results.
\cref{sec:genlam} describes a recursive algorithm for finding all possible examples of a $\PSCA(n,k,\lambda)$ without repeated elements, for arbitrary $\lambda \ge 1$.
\cref{sec:cosets} identifies a group-based structure shared by many examples of perfect sequence covering arrays, and modifies the search algorithm by prescribing this structure.
This allows us to determine new values or bounds for $g(n,k)$ for several pairs $(n,k)$. It also allows us to find examples of new parameter sets $(n,k,\lambda)$ for which a $\PSCA(n,k,\lambda)$ exists, providing evidence for a positive answer to a question posed by Charlie Colbourn (personal communication, Sept.\ 2021): does the existence of a $\PSCA(n,k,\lambda)$ imply the existence of a $\PSCA(n,k,\lambda+1)$?
\cref{sec:open} presents open problems arising from our results. 

The results of this paper are largely based on the Master's thesis of the first author \cite{na-masters}, which contains additional examples and visualizations.

\section{Previous results for $g(n,k)$}
\label{sec:prev}
In this section, we summarize previous results for the value of $g(n,k)$, including constructions, combinatorial nonexistence results, computer nonexistence results, and asymptotic results.

We begin with two trivial constructions.
\begin{lemma}
\label{lem:trivial}
Let $n \ge 2$. Then $g(n,2) = g(n,n) = 1$.
\end{lemma}

\begin{proof}
The set $\{1 2 \cdots n, \,\, n \cdots 2 1\}$ is trivially a $\PSCA(n,2,1)$, and so $g(n,2)=1$.
The set $S_n$ is trivially a $\PSCA(n,n,1)$, and so $g(n,n)=1$.
\end{proof}

The following composition construction is also trivial.
\begin{lemma}
\label{lem:union}
Suppose there exists a $\PSCA(n,k,\lambda)$ and a $\PSCA(n,k,\mu)$. Then their multiset union is a $\PSCA(n,k,\lambda+\mu)$.
\end{lemma}

The following two bounds are straightforward to prove.
\begin{lemma}
\label{lem:modify}
\mbox{}
\begin{enumerate}[$(i)$]
\item
Let $k \le n-1$. Then $g(n,k) \ge g(n-1,k)$.

\item
Let $k \ge 2$. Then $g(n,k) \ge \frac{1}{k}\,g(n,k-1)$.
\end{enumerate}
\end{lemma}

\begin{proof}
For $(i)$, delete the symbol $n$ from each sequence of a~$\PSCA(n,k,\lambda)$ to give a $\PSCA(n-1,k,\lambda)$.
For $(ii)$, regard a $\PSCA(n,k,\lambda)$ as a $\PSCA(n,k-1,k\lambda)$. 
\end{proof}

The following result was proved in terms of perfect codes capable of correcting single deletions.

\begin{theorem}[Levenshtein {\cite[Thm. 3.1]{Lev91}}]
\label{thm:lev}
The set $S_n$ can be partitioned into $n$ sets of sequences, each of which is a $\PSCA(n,n-1,1)$. Therefore $g(n,n-1) = 1$.
\end{theorem}

Levenshtein \cite[p.~140]{Lev90} conjectured in 1990 that the only values of $k$ for which $g(n,k) = 1$ are those provided by \cref{lem:trivial} and \cref{thm:lev}, namely 2, $n-1$, and~$n$. 
This was disproved by the following result.

\begin{proposition}[Mathon 1991, reported in {\cite[p.~191]{MT99}}]
\label{prop:641}
There exists a $\PSCA(6,4,1)$.
\end{proposition}

Mathon and van Trung showed by hand that $g(5,3)>1$ \cite[Thm~3.2]{MT99} (see \cite[p.~13]{na-masters} for a minor correction to the proof), and established the following nonexistence results by computer search.

\begin{proposition}[Mathon and van Trung {\cite[Sects.~4~\&~6]{MT99}}]
\label{prop:g74g75g86}
We have $g(7,4)>1$ and $g(7,5)>1$ and $g(8,6)>1$.
\end{proposition}
\noindent
Mathon and van Trung concluded that $4$ might be the only value of $k$ for which Levenshtein's conjecture fails. We express their revised conjecture in the following form, using \cref{lem:modify}~$(i)$.

\begin{conjecture}[Mathon and van Trung {\cite[p.~198]{MT99}}]
\label{conj:mvtequiv}
Let $k \not \in \{2,4\}$. Then~$g(k+2,k)>1$.
\end{conjecture}
\noindent        
More than twenty years after publication of \cite{MT99}, the smallest open case of \cref{conj:mvtequiv} remains~$k=7$.

The search result $g(7,4)>1$ stated in \cref{prop:g74g75g86} was established in 2004 via an elegant combinatorial proof that does not appear to have been widely recognized outside of the published context of perfect deletion-correcting codes. We therefore rephrase it here.
                    
\begin{theorem}[Klein {\cite[Thm 3.2]{Kle04}}]
\label{thm:klein}
We have $g(7,4)>1$.
\end{theorem}

\begin{proof}
Suppose, for a contradiction, that $P$ is a $\PSCA(7,4,1)$. We may assume after relabelling that $P$ contains the sequence~$1234567$. 
Let $T$ be the set of elements in $P \setminus \{1234567\}$ that contain one of the $3$-subsequences in the set
\[
U = \{124, 134, 234\},
\]
and let $T'$ be the set of elements in $P \setminus \{1234567\}$ that contain one of the $4$-subsequences in the set 
\[
U' = \{3124, 1324, 1243, 2134, 1342, 2314, 2341\}.
\]
It is easy to check that~$T=T'$.
By the PSCA property, the set $P \setminus \{1234567\}$ covers each element of $U'$, 
so there are at least $|U'| = 7$ elements in $T' = T$.

Now in the sequence $1234567 \in P$, each of the symbols 5, 6, 7 occurs after each of the $3$-subsequences in~$U$.
Therefore in every element of $T$, each symbol $5,6,7$ occurs before the symbol~$4$ (otherwise $P$ would cover some $4$-subsequence more than once).
Since there are at least $7$ elements of $T$, but there are only $3! < 7$ ways to order the symbols $5,6,7$ to occur before the symbol~$4$, we conclude that
there are at least two elements of $T$ covering the same $4$-subsequence (formed from some permutation of symbols $5,6,7$ followed by the symbol~$4$). This gives the required contradiction.
\end{proof}
            
The following nonexistence result was proved using matrix rank arguments and by reference to results on covering arrays such as~\cite{Col11}.

\begin{theorem}[Chee et al. {\cite[Thm 2.3]{CCHZ13}}]
\label{thm:2kk>1}
Let $k \ge 3$. Then $g(2k,k)>1$. 
\end{theorem}

Although our interest in this paper is in determining the exact value of $g(n,k)$ for small $n$ and~$k$, we summarize in \cref{thm:yus1,thm:yus2} below the best known asymptotic bounds on the growth rate of $g(n,k)$ as $n$ and~$k$ grow. 
These results improve on previous asymptotic results for completely scrambling sets \cite{Spe71,Ish96,Fur96,Rad03}.

\cref{thm:yus1} holds for general $k$, and was proved by combining
combinatorial arguments with a result due to Wilson \cite[Thm.~1]{Wil90} on the rank of a set inclusion matrix over a finite field. 

\begin{theorem}[Yuster {\cite[Thm.~1]{Yus20}}]
\label{thm:yus1}
Let $k \ge 4$ be an integer.
\begin{enumerate}[$(i)$]
\item
If $k/2$ is a prime, then for all $n \ge k$ we have
\[
\displaystyle{ g(n,k) \ge \frac{\binom{n}{k / 2} - \binom{n}{k / 2 - 1}}{k!}.}
\]
        
\item
Let $n$ and $k$ grow such that $n \gg k$. Then $g(n,k) > n^{k/2-o_k(1)}$ (where $o_k(1)$ represents a function that approaches $0$ as $k \to \infty$).
\end{enumerate}
\end{theorem}

\cref{thm:yus2} holds for the case $k=3$. The proof of the upper bound arises from a recursive construction that builds a $\PSCA(n^2,3,2(n+1)\lambda)$ from a $\PSCA(n,3,\lambda)$, using a finite affine plane of order~$n$ where $n$ is a power of $3$. 

\begin{theorem}[Yuster {\cite[Thm.~2]{Yus20}}]
\label{thm:yus2}
Let $n \ge 3$. Then $n / 6 \le g(n, 3) \le C n (\log_2 n)^{\log_2 7}$ for some absolute constant~$C$.
\end{theorem}

In 2020, Yuster \cite{Yus20} determined that $g(5,3)=2$ by exhibiting a $\PSCA(5,3,2)$. This was the first exact value of $g(n,k)$ greater than~$1$ to be determined. To describe how this result was found, we introduce two definitions.

\begin{definition}
\label{def:IM}
The \emph{$(n,k)$-incidence matrix} is the $n! \times \frac{n!}{(n-k)!}$ array whose rows are indexed by the elements of $S_n$, whose columns are indexed by the elements of $S_{n,k}$, and whose $(x,y)$ entry is $1$ if $x \in S_n$ covers $y \in S_{n,k}$ and is $0$ otherwise.
\end{definition}
\noindent
Each row sum of the $(n,k)$-incidence matrix is $\binom{n}{k}$ and each column sum is~$\frac{n!}{k!}$.

\begin{definition}
\label{def:rvnseq}
Let $X$ be a multiset with elements in~$S_n$.
The \emph{repetition vector of~$X$} with respect to~$k$, written $\rv_k(X)$, is the sum of the rows of the $(n,k)$-incidence matrix that are indexed by the elements of~$X$.
\end{definition}
\noindent
The vector $\rv_k(X)$ has length $\frac{n!}{(n-k)!}$, and its $j^\text{th}$ entry is the number of sequences in $X$ that cover $k$-subsequence~$j$. The multiset $X$ is therefore a $\PSCA(n,k,\lambda)$ if and only if each entry of $\rv_k(X)$ is~$\lambda$.

We can now describe the $\PSCA(5,3,2)$ found by Yuster. 
We shall reinterpret this example in \cref{sec:motivation}.

\begin{example}[Yuster {\cite[Prop.~3.4]{Yus20}}]
\label{ex:532}
Let 
\[
X = \{12345, 43215, 35214, 14523, 25413, 53412\}.
\]
Then the length $60$ vector $\rv_3(X)$ has four entries $0$ (in the positions indexed by the $3$-subsequences $132$, $231$, $154$, $451$), four entries $2$ (in the positions indexed by the $3$-subsequences $123$, $321$, $145$, $541$), and the remaining $52$ entries~$1$. 

Let $\sigma = 13254 \in S_5$ and write $X\sigma = \{x\sigma : x \in X\}$, where
we follow the convention that the composition of permutations $\pi, \sigma$ given by $\pi \sigma$ represents the action of $\pi$ followed by~$\sigma$. Then the repetition vector of
\[
X \sigma = \{13254, 52314, 24315, 15432, 34512, 42513\}
\]
has the same property as $\rv_3(X)$, but the positions in $\rv_3(X\sigma)$ of the $0$ and $2$ entries are interchanged with those in~$\rv_3(X)$. This ensures that every entry of $\rv_3(X) + \rv_3(X\sigma)$ is~$2$, and therefore that $X \cup X\sigma$ is a $\PSCA(5,3,2)$.
Since $g(5,3)>1$ {\rm \cite[Thm.~3.2]{MT99}}, we conclude that $g(5,3)=2$.
\end{example}
\noindent
The original motivation for this paper was the challenge provided by Yuster's concluding statement \cite[p.~592]{Yus20} that
\begin{quote}
``Proving additional exact values of $g(n, k)$ which are not of unit multiplicity in addition to $g(5, 3)$ also seems challenging''.
\end{quote}

\cref{tab:gvalues} summarizes all previously known exact values of $g(n,k)$ for small $n$ and~$k$. 

\begin{table}[htbp]
\centering
\caption{Previously known exact values of $g(n,k)$ for small $n$ and $k$, and their sources.
(s1): \cref{lem:trivial}.
(s2): \cref{thm:lev}.
(s3): \cref{ex:532}.
(s4): \cref{prop:641}.
(s5): \cref{prop:g74g75g86}.
(s6): \cref{thm:klein}.
(s7): \cref{lem:modify}~$(i)$.
}
\label{tab:gvalues}
\vspace{1em} 
\begin{tabular}{c|*{6}{C{1.7cm}}}
\toprule
\diagbox{$n$}{$k$} 
	& 2		& 3		& 4			& 5		& 6		& 7		\\ \hline \hline
2	& 1 (s1)	& \multicolumn{5}{l}{} \\ \cline{1-2}
3	& 1 (s1)	& 1 (s1)	& \multicolumn{4}{l}{} \\ \cline{1-3}
4	& 1 (s1)	& 1 (s2)	& 1 (s1)		& \multicolumn{3}{l}{} \\ \cline{1-4}
5	& 1 (s1)	& 2 (s3)	& 1 (s2)		& 1 (s1)	& \multicolumn{2}{l}{} \\ \cline{1-5}
6	& 1 (s1)	& $\ge 2$ (s7)	& 1 (s4)		& 1 (s2)	& 1 (s1)	& \\ \cline{1-6}
7	& 1 (s1)	& $\ge 2$ (s7)	& $\ge 2$ (s5,s6)	& $\ge 2$ (s5)	& 1 (s2)	& 1 (s1)	\\ \hline
8	& 1 (s1)	& $\ge 2$ (s7)	& $\ge 2$ (s7)		& $\ge 2$ (s7)	& $\ge 2$ (s5)	& 1 (s2)	\\ \hline
9	& 1 (s1)	& $\ge 2$ (s7)	& $\ge 2$ (s7)		& $\ge 2$ (s7)	& $\ge 2$ (s7)	& ?     	\\ \hline
\bottomrule
\end{tabular}
\end{table}

\section{Recursive search algorithm for $\PSCA(n,k,\lambda)$} 
\label{sec:genlam}
In this section, we describe a recursive algorithm for finding all possible examples of a $\PSCA(n,k,\lambda)$ without repeated elements, for arbitrary $\lambda \ge 1$. The algorithm is a tree search that attempts to build a $\PSCA(n,k,\lambda)$ one $n$-sequence at a time without covering any $k$-subsequence more than $\lambda$ times, backtracking when this is not possible. 
Although we have followed Yuster \cite[p.~586]{Yus20} in allowing a perfect sequence covering array to be a multiset, we consider that in many respects it is more natural to restrict the definition of a $\PSCA(n,k,\lambda)$ to a set and have therefore formulated the algorithm to exclude repeated elements. For $\lambda = 1$, this restriction is redundant: if the algorithm terminates without output for parameters $(n,k,1)$, then we can conclude that no $\PSCA(n,k,1)$ exists and so $g(n,k)>1$.

\subsection{Idea of algorithm}
Let $A = (a_{x,y})$ be the $(n,k)$-incidence matrix (where $x \in S_n$ and $y \in S_{n,k}$).
We wish to construct a $\PSCA(n,k,\lambda)$ by finding a $(k!\lambda)$-subset $X$ of $S_n$ for which each entry of $\rv_{k}(X)$ is~$\lambda$, which is equivalent by \cref{def:IM} to
\[
\sum_{x \in X} a_{x,y} = \lambda \quad \mbox{for each~$y$}.
\]
We initialize $X$ to be empty. We then add one element of $S_n$ to $X$ at a time, subject to the condition that
\begin{equation}
\label{eqn:condre}
\sum_{x \in X} a_{x,y} \le \lambda \quad \mbox{for each~$y$}.
\end{equation}
The algorithm succeeds in finding a $\PSCA(n,k,\lambda)$ if $|X|$ reaches~$k!\lambda$. If it is not possible to add an element of $S_n$ to $X$ subject to~\eqref{eqn:condre}, we backtrack.

We keep track of two sets, $Y$ and $L$.
The set $Y$ contains the $k$-subsequences not yet covered $\lambda$ times by the $n$-sequences in~$X$, namely the values $y$ for which $\sum_{x \in X} a_{x,y} < \lambda$.
The set $L$ contains the candidates for enlarging $X$, namely the $n$-sequences that do not cover a $k$-subsequence already covered $\lambda$ times.
We also keep track of the repetition vectors $R = \rv_{k}(X) = (\sum_{x \in X} a_{x,y}) = (r_y)$ 
and $M = \rv_{k}(L) = (\sum_{\ell \in L} a_{l,y}) = (m_y)$, although only at the positions $y$ indexed by~$Y$.
We seek to enlarge $X$ so that the vector $R$ attains the value $\lambda$ in every position. 
Each position of the vector $M$ contains the largest amount by which $R$ can be increased in that position (if every candidate in $L$ were added to~$X$).

Suppose that $X$ currently contains fewer than $k!\lambda$ elements.
If there is a $k$-subsequence that cannot be covered $\lambda$ times, even by adding every candidate in $L$ to~$X$ (that is, $r_y+m_y < \lambda$ for some $y \in Y$), then we terminate the branch early and backtrack.
Otherwise, we find the set $Y'$ of $y \in Y$ for which $r_y$ attains its maximum value (that is, the $k$-subsequences that are not yet covered $\lambda$ times but are closest to being so).
At the next iteration of the algorithm, we choose one $y \in Y'$ and recurse by attempting to add to $X$ each of the elements of 
$\{\ell \in L : a_{l,y} = 1 \}$ in turn (each such addition causing the $k$-subsequence $y$ to be covered one more time).
In order to reduce the number of tree branches that must be searched, we choose a value of $y \in Y'$ for which 
$|\{\ell \in L : a_{l,y} = 1 \}| = m_y$ is minimized: call this value~$y^*$.
Although the choice of $y^*$ is restricted to the subset $Y'$ of $Y$, and might not be unique within $Y'$, the algorithm is exhaustive over the possibilities for incrementing the value of~$r_{y^*}$. 
Since we require each entry of $r_y$ for $y \in Y$ to eventually reach the value $\lambda$, 
the algorithm finds every possible example of a $\PSCA(n,k,\lambda)$ regardless of the sequence of values~$y^*$ chosen.

For each $x \in L$ in turn for which $a_{x,y^*} = 1$, we update the sets 
$X, Y, L$ and the repetition vectors $R, M$ and then recurse. 
We update $X$ by adding the $n$-sequence~$x$.
We update $Y$ by removing all $k$-subsequences newly covered $\lambda$ times.
We update $R$ by adding the row of $A$ indexed by $x$ to it.
We update $L$ by removing $x$ (so that the same $n$-sequence $x$ cannot be added to $X$ a second time) and by removing each $n$-sequence covering a $k$-subsequence that is newly covered $\lambda$ times (because the later inclusion of such an $n$-sequence in $X$ would violate~\eqref{eqn:condre}).
We update $M$ by subtracting the rows of $A$ that have just been removed from~$L$.

Pseudocode implementing this search procedure is given in \cref{alg:genPSCA}.
By calling the procedure SEARCH with 
$X = \emptyset$ and $Y = S_{n,k}$ and $R = (0,\dots,0)$ and $L = S_n$ and $M = (\frac{n!}{k!}, \dots, \frac{n!}{k!})$, we obtain every possible $\PSCA(n,k,\lambda)$ without repeated elements.
Furthermore, we may assume after relabelling symbols that the perfect sequence covering array contains the sequence $12 \cdots n$.
We therefore replace, for the first iteration of \cref{alg:genPSCA}, the set $L(y^*)$ at Line $18$ by the single-element set $\{12 \cdots n\}$.

\begin{algorithm}
\caption{Tree search for $\PSCA(n,k,\lambda)$ by backtracking}
\label{alg:genPSCA}
\begin{algorithmic}[1]
\Require PSCA parameters $(n,k,\lambda)$ and $A = \big(a_{x,y} : x \in S_n, \, y \in S_{n,k} \big)$ = $(n,k)$-incidence matrix 
\Procedure{SEARCH}{$X, Y, R, L, M$}
    \State \Comment{\%~$X$ = sequences of partial $\PSCA(n,k,\lambda)$ $P$}
    \State \Comment{\%~write $(r_y) = \rv_{k}(X) = (\sum_{x \in X}~a_{x,y})$ = sum of rows of $A$ indexed by~$X$}
    \State \Comment{\%~$Y = \{y \in S_{n,k} : r_y < \lambda\}$ = $k$-subsequences not yet covered $\lambda$ times by $P$}
    \State \Comment{\%~$R = (r_y : y \in Y)$ = entries of $\rv_{k}(X)$ indexed by $Y$}
    \State \Comment{\%~$L = \{\ell \not \in X: a_{\ell,y} = 0$ for all $y \not \in Y\}$ = candidates for enlarging~$X$}
    \State \Comment{\%~write $(m_y) = \rv_{k}(L) = (\sum_{\ell \in L} a_{\ell,y})$ = sum of rows of $A$ indexed by~$L$}
    \State \Comment{\%~$M = (m_y: y \in Y)$ = entries of $\rv_{k}(L)$ indexed by~$Y$}

    \If {$|X| = k!\lambda$}
        \State record $X$ as a $\PSCA(n,k,\lambda)$.
        \State \Return
    \EndIf

    \If{$r_y + m_y < \lambda$ for some $y \in Y$}
        \State \Return  \Comment{\quad \% terminate branch early}
    \EndIf

    \State let $Y'$ be the set of $y \in Y$ for which $r_y$ attains its maximum value.
    \State choose an arbitrary $y = y^* \in Y'$ for which $m_y$ attains its minimum value.
    \State $L(y^*) \coloneqq \{ x \in L : a_{x,y^*} = 1\}$
    \For {\textbf{each} $x \in L(y^*)$}
        \State $X_{\text{new}} \coloneqq X \cup \{x\}$
        \State $Y_{\text{new}} \coloneqq Y \setminus \{y : r_y + a_{x,y} = \lambda\}$
        \State $R_{\text{new}} \coloneqq (r_y + a_{x,y} : y \in Y_\text{new})$
        \State $B \coloneqq \{x\} \cup \{\ell \in L:$ $a_{\ell,y} = 1$ for at least one $y \in Y \setminus Y_\text{new}\}$ 
        \State $L_{\text{new}} \coloneqq L \setminus B$
        \State $M_{\text{new}} \coloneqq (m_y - \sum_{\ell \in B}a_{\ell,y} : y \in Y_\text{new})$
        \State SEARCH$(X_{\text{new}}, Y_{\text{new}}, R_{\text{new}}, L_{\text{new}}, M_{\text{new}})$
    \EndFor
\EndProcedure
\end{algorithmic}
\end{algorithm}

\subsection{Discussion}
The nonexistence results for a $\PSCA(n,k,1)$ given in \cref{prop:g74g75g86} were obtained by Mathon and van Trung using a search algorithm for a directed $t$-packing \cite[p.~163]{MT99}, that in turn relies on a algorithm due to Mathon~\cite{Mat97} for finding spreads in an incidence structure.
A key feature of this algorithm is that each successive iteration removes rows and columns of the incidence matrix describing the incidence structure, leading to reduced time and space complexity.
Mathon noted \cite[p.164]{Mat97} that
``An actual implementation of this algorithm requires good data structures to facilitate fast and efficient computations in, and updating of the various point and line sets'', without explicitly describing these data structures.

\cref{alg:genPSCA} for finding a $\PSCA(n,k,\lambda)$ without repeated elements is inspired by Mathon's algorithm, and reduces to a broadly equivalent form in the special case $\lambda = 1$.
In particular, the idea of improving efficiency by restricting the search to the subset $Y'$ of $Y$ in \cref{alg:genPSCA} is taken directly from Mathon's paper~\cite{Mat97}.
The cases $\lambda > 1$ do not have a corresponding form in the context of a directed $t$-packing and so were not considered in~\cite{Mat97}. 
\cref{alg:genPSCA} also contains a feature not described in \cite{Mat97} or \cite{MT99} that leads to a significant speed advantage for all cases $\lambda \ge 1$: keeping track of the vector $M$ and passing it as a recursion parameter, which avoids having to calculate $|Y|$ column sums over $|L|$ rows when carrying out the early termination test at Line~13.
The calculations in Lines 13, 16, 17 can be accomplished in linear time with a single pass through the positions indexed by~$Y$.

The space complexity of \cref{alg:genPSCA} is determined by the representation of the $(n,k)$-incidence matrix. 
There is no need to store this matrix explicitly as an $n! \times \frac{n!}{(n-k)!}$ matrix over $\{0,1\}$: it is sufficient to store the positions of the 1 entries in each row, and the positions of the 1 entries in each column.

\subsection{New values for $g(n,k)$}
\cref{alg:genPSCA} finds the following examples of a $\PSCA(6,3,2)$ and a~$\PSCA(7,3,2)$. 
By reference to \cref{tab:gvalues}, this gives the new results $g(6,3)=g(7,3)=2$.

\begin{proposition}
\label{prop:newg-gen}
\mbox{}
\begin{enumerate}[$(i)$]
\item
The following $12$ sequences form a $\PSCA(6,3,2)$
\[
    \begin{array}{*{6}{c}}
123456 & 154326 & 216543 & 245613 & 354162 & 361452 \\
423165 & 461325 & 516234 & 532614 & 632541 & 645231
    \end{array}
\]
and therefore $g(6,3)=2$.

\item
The following $12$ sequences form a $\PSCA(7,3,2)$
\[
    \begin{array}{*{6}{c}}
1234567 & 1573426 & 3275641 & 3617524 & 4261735 & 4756123 \\
5164327 & 5243176 & 6257314 & 6345721 & 7216453 & 7431625
    \end{array}
\]
and therefore $g(7,3)=2$.
\end{enumerate}
\end{proposition}

\section{PSCA$(n,k,\lambda)$ as union of cosets of a prescribed subgroup}
\label{sec:cosets}
In this section, we identify an algebraic structure shared by many examples of perfect sequence covering arrays. We then modify \cref{alg:genPSCA} to search for perfect sequence covering arrays having this prescribed structure, thereby determining new values or bounds for $g(n,k)$ for several pairs~$(n,k)$. The algorithm also finds examples of new parameter sets $(n,k,\lambda)$ for which a $\PSCA(n,k,\lambda)$ exists. As in \cref{sec:genlam}, we restrict the search to perfect sequence covering arrays without repeated elements.

\subsection{Motivation}
\label{sec:motivation}
The $\PSCA(5,3,2)$ given in \cref{ex:532} was constructed by Yuster as $X \cup X\sigma$ for a $6$-subset $X$ of~$S_5$ and a permutation $\sigma \in S_5$.
We can equivalently interpret this perfect sequence covering array as the union of six left cosets of the order $2$ subgroup $\lsr$ of $S_5$, by reading the following table not by rows (as $X \cup X\sigma$) but by columns (as $\bigcup_{x \in X} \, x \lsr$, where $X$ is a set of left coset representatives for~$\lsr$).

\begin{center}
\begin{tabular}{c||cccccc}
    \multicolumn{7}{c}{$\PSCA(5,3,2)$ with $\sigma = 13254$} \\
    \toprule
    & $\lsr$ & $43215 \, \lsr$ & $35214 \, \lsr$ & $14523 \, \lsr$ & $25413 \, \lsr$ & $53412 \, \lsr$ \\ \hline \hline
    $X$       &  12345 &   43215    &   35214    &   14523    &   25413    &   53412    \\
    $X\sigma$ &  13254 &   52314    &   24315    &   15432    &   34512    &   42513    \\
    \bottomrule
\end{tabular}
\end{center}

We can likewise reinterpret the $\PSCA(6,3,2)$ and $\PSCA(7,3,2)$ of \cref{prop:newg-gen}, found using \cref{alg:genPSCA}, as the union of six left cosets of a suitable order $2$ subgroup $\lsr$ of~$S_6$ and $S_7$, respectively.

\begin{center}
\begin{tabular}{c||cccccc}
    \multicolumn{7}{c}{$\PSCA(6,3,2)$ with $\sigma = 154326$} \\
    \toprule
    & $\lsr$ & $216543 \, \lsr$ & $354162 \, \lsr$ & $461325 \, \lsr$ & $532614 \, \lsr$ & $645231 \, \lsr$ \\
    \hline
    \hline
    $X$       & 123456 & 216543 & 354162 & 461325 & 532614 & 645231 \\
    $X\sigma$ & 154326 & 516234 & 423165 & 361452 & 245613 & 632541 \\
    \bottomrule
\end{tabular}
\end{center}

\begin{center}
\begin{tabular}{c||cccccc}
    \multicolumn{7}{c}{$\PSCA(7,3,2)$ with $\sigma = 4261735$} \\
    \toprule
    & $\lsr$ & $1573426 \, \lsr$ & $3275641 \, \lsr$ & $3617524 \, \lsr$ & $5164327 \, \lsr$ & $5243176 \, \lsr$ \\ \hline \hline
    $X$       & 1234567 & 1573426 & 3275641 & 3617524 & 5164327 & 5243176 \\
    $X\sigma$ & 4261735 & 4756123 & 6257314 & 6345721 & 7431625 & 7216453 \\
    \bottomrule
\end{tabular}
\end{center}

This motivates us to seek a $\PSCA(n,k,\lambda)$ as the union of $\frac{k!\lambda}{|G|}$ distinct cosets of a nontrivial subgroup $G$ of $S_n$, where the parameter $\lambda$ need not take the value~$2$; the subgroup $G$ need not be cyclic, nor have order~$2$;
and we can choose either right cosets of a subgroup or left cosets.
However, we shall see in \cref{sec:cosets} that the use of left cosets leads to a simplification in the search (using conjugacy classes), and a richer existence pattern (involving a larger subgroup $G$).

We note that a single $\PSCA(n,k,\lambda)$ can admit more than one representation as a union of distinct cosets of a subgroup of~$S_n$. 
For example, the $\PSCA(6,3,2)$ represented above as six left cosets of the order~$2$ subgroup $\langle 154326 \rangle$
can also be represented as the entire order~$12$ subgroup $\langle 154326, 216543 \rangle \cong D_{12}$.
Likewise, the $\PSCA(7,3,2)$ represented above as 
six left cosets of the order~$2$ subgroup $\langle 4261735 \rangle$ can also by represented as
six right cosets of the order~$2$ subgroup $\langle 3617524 \rangle$, and as
two left cosets of the order~$6$ subgroup $\langle 3275641, 4261735 \rangle \cong S_3$.

We remark that several aspects of our approach can be recognized in Mathon and van Trung's work~\cite{MT99}. We note in particular the following refinement of \cref{prop:641}, found by computer search. We write $D_{2n}$ for the dihedral group of order~$2n$.

\begin{proposition}[Mathon and van Trung {\cite[Thm 4.1]{MT99}}]
\label{prop:641group}
Up to isomorphism, there are exactly two examples $P_1, P_2$ of a $\PSCA(6,4,1)$:
\begin{enumerate}[$(i)$]
\item 
the $24$ sequences of $P_1$ comprise a subgroup $G_1 \cong S_4$ of~$S_6$; the automorphism group of $P_1$ is isomorphic to~$S_4$.

\item
the $24$ sequences of $P_2$ comprise a union of three right cosets of a subgroup $G_2 \cong D_8$ of~$S_6$;
the automorphism group of $P_2$ is isomorphic to~$D_8$. (See {\rm \cite[p.~29]{na-masters}} for a correction to the listing of $P_2$ in {\rm \cite[p.~192]{MT99}}.)
\end{enumerate}
\end{proposition}

\subsection{Left or right cosets}
\label{sec:LR}
We next show how, for a perfect sequence covering array comprising a union of distinct cosets of a nontrivial subgroup, there is a fundamental distinction between left and right cosets.
We firstly note that the right action of a permutation on a subset $X$ of $S_n$ permutes the entries of the repetition vector~$\rv_k(X)$.

\begin{lemma}
\label{lem:right}
Let $X$ be a subset of $S_n$ and let $\sigma \in S_n$. Then the vector $\rv_k(X\sigma)$ is obtained by permuting the entries of the vector $\rv_k(X)$.
\end{lemma}

\begin{proof}
Under the convention that the permutation composition $\pi \sigma$ represents the action of $\pi$ followed by~$\sigma$,
the right action of $\sigma$ on the set $X$ permutes the symbols in~$[n]$ and so permutes the elements of~$S_{n,k}$. The result follows by \cref{def:rvnseq}.
\end{proof}
\noindent
It follows that the right action of a permutation $\sigma \in S_n$ maps one $\PSCA(n,k,\lambda)$ to another.
\begin{corollary}
\label{cor:relabel}
    Suppose $P$ is a $\PSCA(n,k,\lambda)$, and let $\sigma \in S_n$. Then $P\sigma$ is also a $\PSCA(n,k,\lambda)$.
\end{corollary}

Subgroups $G$ and~$H$ of $S_n$ are \emph{conjugate} in~$S_n$ if $H=\sigma G\sigma^{-1}$ for some $\sigma \in S_n$. Conjugacy is an equivalence relation on the set of subgroups of $S_n$, and the equivalence class of $G$ under conjugation is the \emph{conjugacy class} $\Cl(G) := \{\sigma G\sigma^{-1}: \sigma \in S_n\}$.
Consider searching over all nontrivial subgroups $G$ of~$S_n$ and all sets $\Ra$ (of cardinality $\frac{k!\lambda}{|G|}$) of left coset representatives for~$G$ to find a perfect sequence covering array of the form $\bigcup_{\pi \in \Ra} \pi G$.
We now use \cref{cor:relabel} to show that it is sufficient to restrict attention to a single representative $G$ from each conjugacy class of subgroups of~$S_n$. This drastically reduces the required computation for an exhaustive search. For example, $S_7$ contains 11299 nontrivial subgroups, but only 95 nontrivial conjugacy classes of subgroups.

\begin{theorem}
\label{thm:leftconj}
Let $G$ be a subgroup of~$S_n$ and let $H \in \Cl(G)$. Suppose there is a $\PSCA(n,k,\lambda)$ that is a union of distinct left cosets of~$H$. Then there is a $\PSCA(n,k,\lambda)$ which is a union of distinct left cosets of~$G$.
\end{theorem}

\begin{proof}
Let $P$ be a $\PSCA(n,k,\lambda)$ that can be written as
\begin{equation}
\label{eqn:union}
P = \underset{\pi \in \Ra}{\bigcup} \pi H
\end{equation}
for a set $\Ra$ of left coset representatives for~$H$.
Since $H \in \Cl(G)$, for some $\sigma \in S_n$ we can write

\begin{align}
P &= \underset{\pi \in \Ra}{\bigcup} \pi (\sigma G\sigma^{-1})
\notag \\
  &= \bigg( \underset{\mu \in \Ra \sigma}{\bigcup} \mu G \bigg) \sigma^{-1}.
  \notag
\end{align}

Then $P\sigma = \bigcup_{\mu \in \Ra \sigma} \mu G$ is a union of distinct left cosets of~$G$, and is a~$\PSCA(n,k,\lambda)$ by \cref{cor:relabel}.
\end{proof}

Consider instead searching for a perfect sequence covering array as a union $\bigcup_{\sigma \in \Ra} G \sigma$ of distinct right cosets of~$G$.
\cref{thm:leftconj} no longer holds when we replace left cosets by right cosets, because it relies on \cref{cor:relabel} which fails when the right action of $\sigma$ on $P$ is replaced by the left action.
Therefore an exhaustive search must consider all nontrivial subgroups of~$S_n$. 
However, there is now a worthwhile simplification in that we may assume the subgroup $G$ itself is one of the right cosets contained in the perfect sequence covering array $\bigcup_{\sigma \in \Ra} G\sigma$: let $\mu \in \Ra$, and note from \cref{cor:relabel} that $\Big(\bigcup_{\sigma \in \Ra} G\sigma \Big)\mu^{-1} = \bigcup_{\sigma \in \Ra} G (\sigma \mu^{-1})$ is a perfect sequence covering array, and that it contains the right coset~$G (\mu \mu^{-1}) = G$.

\subsection{Algorithm description}
\label{sec:alg-desc}
We can now describe an algorithm for finding all possible examples of a $\PSCA(n,k,\lambda)$ that is a union of $\frac{k!\lambda}{|G|}$ distinct cosets of a prescribed nontrivial subgroup $G$ of~$S_n$. The algorithm attempts to build a $\PSCA(n,k,\lambda)$ one coset of $G$ at a time without covering any $k$-subsequence more than $\lambda$ times. If it terminates without output then there is no $\PSCA(n,k,\lambda)$ that is a union of distinct cosets of~$G$.

The algorithm follows the same principles as \cref{alg:genPSCA}, but operates on the following compressed version of the $(n,k)$-incidence matrix in which the $|G|$ repetition vectors indexed by the sequences of a coset are replaced by their sum, because the entire coset is either contained or not contained in the perfect sequence covering array.
The advantage of this approach is that prescribing the subgroup~$G$ reduces the maximum search depth by a factor of~$|G|$. This gives a dramatic speed improvement, even for~$|G| = 2$. 

\begin{definition}
\label{def:IMgroup}
Let $G$ be a subgroup of $S_n$, and let $\Ra$ be a complete set of left (or right) coset representatives for $G$ in~$S_n$.
The \emph{left (or right) $(n,k)$-incidence matrix for~$G$} is the $\frac{n!}{|G|} \times \frac{n!}{(n-k)!}$ array over $\{0,1,\dots,|G|\}$ whose rows are indexed by $\Ra$, whose columns are indexed by the elements of $S_{n,k}$, and whose $(x,y)$ entry is the number of times the coset $xG$ (or $Gx$) covers $y \in S_{n,k}$.
\end{definition}

Pseudocode implementing the search procedure is given in \cref{alg:subgroup} (which reduces to \cref{alg:genPSCA} if we take $G$ to be the trivial subgroup).
\cref{alg:genPSCA} keeps track of a shrinking set $L$ of rows and a shrinking set $Y$ of columns of the $(n,k)$-incidence matrix; \cref{alg:subgroup} does the same, but in relation to the $(n,k)$-incidence matrix for~$G$. 
This requires the following modifications to \cref{alg:genPSCA}, because the entries $a_{x,y}$ of the latter matrix are no longer restricted to~$\{0,1\}$:

\begin{description}
\item Line 9.
The target size of $|X|$ is the number $\frac{k! \lambda}{|G|}$ of cosets, rather than the number $k!\lambda$ of $n$-sequences.

\item Line 18.
The entries of the $(n,k)$-incidence matrix for $G$ lie in $\{0,1,\dots,|G|\}$ rather than~$\{0,1\}$, so the test $a_{x,y^*} = 1$ is replaced by $a_{x,y^*} \ne 0$. 

\item Line 6.
The set $L$ of candidates for enlarging $X$ is more constrained, because inclusion of a row containing an entry $a_{\ell,y} > 1$ could cause the sum $a_{\ell,y} + r_y$ to exceed~$\lambda$. We therefore impose that $a_{\ell,y} + r_y \le \lambda$ over all values of $y$ (not just those lying outside~$Y$). We can rewrite $L$ as specified in Line 6 as 
\[
    \{\ell \not \in X: a_{\ell,y} = 0 \mbox{ for all } y \not \in Y\} \cap \{\ell \not \in X: a_{\ell,y} + r_y \le \lambda \mbox{ for all } y \in Y \},
\]
to see that the set $L$ in \cref{alg:subgroup} is a subset of the set $L$ in \cref{alg:genPSCA}.

\item Line 23.
The set $B$ used to update $L$ at Line~24 follows from the definition of $L$ at Line 6 and the updated value of $R$ at Line~22. 
\end{description}

\begin{algorithm}
\caption{Tree search for union-of-cosets $\PSCA(n,k,\lambda)$ by backtracking}
\label{alg:subgroup}
\begin{algorithmic}[1]
\Require PSCA parameters $(n, k, \lambda)$ and $A = \big(a_{x,y} : x \in \Ra,\, y \in S_{n,k} \big)$ = left or right $(n,k)$-incidence matrix for a subgroup $G$ of~$S_n$ and complete set of coset representatives $\Ra$, where $|G|$ divides~$k! \lambda$. 
\Procedure{SEARCH}{$X, Y, R, L, M$}
    \State \Comment{\%~$X$ = coset representatives of partial $\PSCA(n,k,\lambda)$ $P$}
    \State \Comment{\%~write $(r_y) = \rv_k(X) = (\sum_{x \in X}~a_{x,y})$ = sum of rows of $A$ indexed by~$X$}
    \State \Comment{\%~$Y = \{y \in S_{n,k}: r_y < \lambda\}$ = $k$-subsequences not yet covered $\lambda$ times by $P$}
    \State \Comment{\%~$R = (r_y : y \in Y)$ = entries of $\rv_k(X)$ indexed by~$Y$}
    \State \Comment{\%~$L = \{\ell \not \in X: a_{\ell,y} + r_y \le \lambda$ for all $y\}$ = candidates for enlarging $X$} 
    \State \Comment{\%~write $(m_y) = \rv_k(L) = (\sum_{\ell \in L} a_{\ell,y})$ = sum of rows of $A$ indexed by~$L$}
    \State \Comment{\%~$M = (m_y: y \in Y)$ = entries of $\rv_k(L)$ indexed by~$Y$}

    \If {$|X| =\frac{k!\lambda}{|G|}$}
        \State record $X$ as a set of coset representatives for a $\PSCA(n,k,\lambda)$.
        \State \Return
    \EndIf
                    
    \If{$r_y + m_y < \lambda$ for some $y \in Y$}
        \State \Return  \Comment{\quad \% terminate branch early}
    \EndIf

    \State let $Y'$ be the set of $y \in Y$ for which $r_y$ attains its maximum value.
    \State choose an arbitrary $y = y^* \in Y'$ for which $m_y$ attains its minimum value.
    \State $L(y^*) \coloneqq \{ x \in L : a_{x,y^*} \ne 0\}$
    \For {\textbf{each} $x \in L(y^*)$}
        \State $X_{\text{new}} \coloneqq X \cup \{x\}$
        \State $Y_{\text{new}} \coloneqq Y \setminus \{y : r_y + a_{x,y} = \lambda\}$
        \State $R_{\text{new}} \coloneqq (r_y + a_{x,y} : y \in Y_\text{new})$
        \State $B \coloneqq \{x\} \cup \{\ell \in L:$ $a_{\ell,y} + r_y + a_{x,y} > \lambda$ for some $y \in Y\}$ 
        \State $L_{\text{new}} \coloneqq L \setminus B$
        \State $M_{\text{new}} \coloneqq (m_y - \sum_{\ell \in B}a_{\ell,y} : y \in Y_\text{new})$
        \State SEARCH$(X_{\text{new}}, Y_{\text{new}}, R_{\text{new}}, L_{\text{new}}, M_{\text{new}})$
    \EndFor
    \EndProcedure
\end{algorithmic}
\end{algorithm}

The procedure SEARCH of \cref{alg:subgroup} searches for a $\PSCA(n,k,\lambda)$ comprising a union of $\frac{k!\lambda}{|G|}$ distinct cosets of a single prescribed nontrivial subgroup $G$ of~$S_n$, where $|G|$ divides~$k! \lambda$. 
In the case of left cosets, it is sufficient to search over a single representative $G$ of each conjugacy class of nontrivial subgroups of~$S_n$ (see \cref{sec:LR}). For each such subgroup $G$, we may exclude from the initial candidate set $L$ each coset representative that indexes a row of the left $(n,k)$-incidence matrix for $G$ containing some entry greater than~$\lambda$, and initialize $M$ accordingly. 
To determine whether a $\PSCA(n,k,\lambda)$ exists using this procedure, it is most efficient to examine the set of suitable subgroups $G$ for \cref{alg:subgroup} in decreasing order of~$|G|$, because a larger value of $|G|$ gives a more dramatic speed improvement over \cref{alg:genPSCA}. We do not take $G$ to the trivial group, because then \cref{alg:subgroup} reduces to \cref{alg:genPSCA} and, even if a perfect sequence covering array is found, no additional structure is identified.

In the case of right cosets, we must instead search over all subgroups $G$ of $S_n$, although we may take the initial set $X$ to be $\{1_G\}$ and initialize $Y$, $R$, $L$, $M$ accordingly (see \cref{sec:LR}). We need examine the subgroup $G$ only if the row indexed by $1_G$ of the right $(n,k)$-incidence matrix for $G$ has all entries at most $\lambda$; if so, then the same holds for all other rows by \cref{lem:right}.

\cref{tab:timing} illustrates the differences in CPU time required to search for all possible examples of a $\PSCA(n,k,\lambda)$ without repeated elements using
\cref{alg:genPSCA}, and using distinct left or right cosets with \cref{alg:subgroup} over all subgroups $G$ of a specified order.
These search times refer to a C implementation on a single core of an Intel Xeon E5-2680, excluding the precalculation time for incidence matrices in GAP \cite{GAP4} for \cref{alg:subgroup} (which carries negligible overhead for larger searches).
Comparison of timings for $(n,k,\lambda) = (6,3,2)$ shows that \cref{alg:subgroup}, when it succeeds, is significantly faster than \cref{alg:genPSCA} even when $|G| = 2$. 
Comparison of timings for $(n,k,\lambda) = (7,3,2)$ shows that \cref{alg:subgroup}, when it succeeds for a larger $|G|$, is dramatically faster than with a smaller~$|G|$.
Comparison of timings for $(n,k,\lambda) = (7,4,2)$ shows that \cref{alg:subgroup} using left cosets can succeed when \cref{alg:subgroup} using right cosets fails for the same $|G|$,
and in that case a successful search using left cosets is significantly faster than using right cosets.
The comparison times for the parameter sets $(7,5,1)$ and $(8,6,1)$ taken from \cite{MT99} refer to exhaustive searches carried out in 1999 on an Ultra SPARCstation 5 to establish \cref{prop:g74g75g86}.

\begin{table}[htbp]
\caption{CPU time to search for all possible examples of a $\PSCA(n,k,\lambda)$}
\label{tab:timing}
\centering
\begin{tabular}{l|ccccl|l}
    \toprule
    method		& cosets 	& $(n,k,\lambda)$   & $|G|$ 	& examples? 	& CPU time	& 1999 search \cite{MT99} \\ \hline
    \cref{alg:genPSCA}	&		& $(7,5,1)$         &      	& no		& 0.1 seconds 	& 5 minutes \\
    \cref{alg:genPSCA}	&		& $(8,6,1)$         &      	& no		& 100 minutes	& 100 hours \\
    \cref{alg:genPSCA}	&		& $(6,3,2)$         &      	& yes		& 3 seconds	& \\
    \cref{alg:genPSCA}	&		& $(7,3,2)$         &      	& yes		& 40 minutes	& \\ \hline
    \cref{alg:subgroup} & left		& $(6,3,2)$         & 2     	& yes		& 0.4 seconds	& \\
    \cref{alg:subgroup} & left 		& $(7,3,2)$         & 2    	& yes		& 50 seconds 	& \\
    \cref{alg:subgroup} & left       	& $(7,3,2)$         & 6    	& yes		& 0.05 seconds 	& \\
    \cref{alg:subgroup} & left        	& $(7,4,2)$         & 6    	& yes		& 3 seconds	& \\ \hline
    \cref{alg:subgroup} & right        	& $(7,4,2)$         & 6    	& no		& 10 seconds	& \\
    \cref{alg:subgroup} & right        	& $(7,4,2)$         & 2    	& yes		& 210 minutes 	& \\
    \bottomrule
\end{tabular}
\end{table}

\subsection{New values and bounds for $g(n,k)$}
\label{sec:newg}
\cref{alg:subgroup} finds the following examples of a $\PSCA(7,4,2)$, a $\PSCA(7,5,4)$, a $\PSCA(8,3,3)$, and a $\PSCA(9,3,4)$ as a union of left cosets. 
By reference to \cref{tab:gvalues}, this gives the new results $g(7,4)=2$ and $g(7,5) \in \{2,3,4\}$ and $g(8,3) \in \{2,3\}$ and $g(9,3) \in \{2,3,4\}$.

\begin{proposition}
\label{prop:newg-subgroup}
\mbox{}
\begin{enumerate}[$(i)$]
\item
%

The following $48$ sequences form a $\PSCA(7,4,2)$ as a union of $8$ left cosets of the subgroup 
$G = \langle 4735621 \rangle \cong C_6$ of~$S_7$:
\[
\begin{array}{c||*{6}{c}}
\text{{\rm left coset}}   & \multicolumn{6}{c}{{\rm sequences}}		  \\ \hline
 1254736\,G &   1254736 & 2517634 & 4765132 & 5126437 & 6472531 & 7641235 \\
 1347256\,G &   1347256 & 2376514 & 4351762 & 5364127 & 6325471 & 7312645 \\
 1352746\,G &   1352746 & 2315674 & 4367152 & 5321467 & 6374521 & 7346215 \\
 1362745\,G &   1362745 & 2345671 & 4327156 & 5371462 & 6314527 & 7356214 \\
 1365472\,G &   1365472 & 2341765 & 4326517 & 5372641 & 6317254 & 7354126 \\
 1462537\,G &   1462537 & 2745136 & 4527631 & 5671234 & 6214735 & 7156432 \\
 1745236\,G &   1745236 & 2671534 & 4156732 & 5462137 & 6527431 & 7214635 \\
 1765234\,G &   1765234 & 2641537 & 4126735 & 5472136 & 6517432 & 7254631 
\end{array}
\]
Therefore $g(7,4)= 2$.

\item

The following $480$ sequences (listed only by reference to cosets) form a $\PSCA(7,5,4)$ as $20$ left cosets of the subgroup 
$G = \langle 7261354,\, 4216537 \rangle \cong S_4$ of~$S_7$:
\[
\begin{array}{*{5}{c}}
1234567\,G & 1236574\,G & 1324576\,G & 1324675\,G & 1325476\,G \\
1325674\,G & 1347265\,G & 1354267\,G & 1357642\,G & 1364275\,G \\
1364572\,G & 1367542\,G & 1374256\,G & 1375624\,G & 1623574\,G \\
1635247\,G & 1637425\,G & 2137465\,G & 2137654\,G & 2163457\,G 
\end{array}
\]

Therefore $g(7,5) \in \{2,3,4\}$.

\item
%

The following $18$ sequences form a $\PSCA(8,3,3)$ as a union of $9$ left cosets of the subgroup 
$G = \langle 85672341 \rangle \cong C_2$ of~$S_8$:
\[
\begin{array}{c||*{2}{c}}
\text{{\rm left coset}}   & \multicolumn{2}{c}{{\rm sequences}}	\\ \hline
 12345678\,G &   12345678 & 85672341 \\
 15468237\,G &   15468237 & 82731564 \\
 17624385\,G &   17624385 & 84357612 \\
 27561843\,G &   27561843 & 54238176 \\
 28461573\,G &   28461573 & 51738246 \\
 31864275\,G &   31864275 & 68137542 \\
 32654187\,G &   32654187 & 65327814 \\
 37461528\,G &   37461528 & 64738251 \\
 47218653\,G &   47218653 & 74581326 
\end{array}
\]
Therefore $g(8,3) \in \{2,3\}$.

\item
%

The following $24$ sequences form a $\PSCA(9,3,4)$ as a union of $4$ left cosets of the subgroup 
$G = \langle 768241593 \rangle \cong C_6$ of~$S_9$:
\[
\begin{array}{c||*{6}{c}}
\text{{\rm left coset}}   & \multicolumn{6}{c}{{\rm sequences}} 			\\ \hline
 123456897\,G &   123456897 & 258714936 & 473165892 & 519627384 & 649572381 & 768241935 \\
 154372968\,G &   154372968 & 217865349 & 461327958 & 526941873 & 675914823 & 742856319 \\
 198426537\,G &   198426537 & 239754186 & 498175632 & 583617294 & 683542791 & 739261485 \\
 318697425\,G &   318697425 & 348592176 & 829436751 & 879135264 & 953784612 & 963281547 
\end{array}
\]
Therefore $g(9,3) \in \{2,3,4\}$.
\end{enumerate}
\end{proposition}

\cref{tab:gvalues-updated} is an updated table of values of $g(n,k)$ for small $n$ and $k$, taking account the new values and bounds from \cref{prop:newg-gen} and \cref{prop:newg-subgroup}.

\begin{table}[htbp]
\centering
\caption{Updated table of $g(n,k)$ for small $n$ and $k$, with new values or bounds in bold. 
(s1): \cref{lem:trivial}.
(s2): \cref{thm:lev}.
(s3): \cref{ex:532}.
(s4): \cref{prop:641}.
(s5): \cref{prop:g74g75g86}.
(s6): \cref{thm:klein}.
(s7): \cref{lem:modify}~$(i)$.
(s8): \cref{prop:newg-gen}.
(s9): \cref{prop:newg-subgroup}.
}
\label{tab:gvalues-updated}
\vspace{1em} 
\begin{tabular}{c|*{5}{c}C{1.7cm}}
\toprule
\diagbox{$n$}{$k$} 
	& 2		& 3				& 4		& 5				& 6		& 7		\\ \hline \hline
2	& 1 (s1)	& \multicolumn{5}{l}{} \\ \cline{1-2}
3	& 1 (s1)	& 1 (s1)			& \multicolumn{4}{l}{} \\ \cline{1-3}
4	& 1 (s1)	& 1 (s2)			& 1 (s1)	& \multicolumn{3}{l}{} \\ \cline{1-4}
5	& 1 (s1)	& 2 (s3)			& 1 (s2)	& 1 (s1)			& \multicolumn{2}{l}{} \\ \cline{1-5}
6	& 1 (s1)	& {\bf 2} (s8)			& 1 (s4)	& 1 (s2)			& 1 (s1)	& 		\\ \cline{1-6}
7	& 1 (s1)	& {\bf 2} (s8)			& {\bf 2} (s9)	& {\bf 2} or {\bf 3} or {\bf 4} (s9)& 1 (s2)	& 1 (s1)	\\ \hline
8	& 1 (s1)	& {\bf 2} or {\bf 3} (s9)	& $\ge 2$ (s7)	& $\ge 2$ (s7)			& $\ge 2$ (s5)	& 1 (s2)	\\ \hline
9	& 1 (s1)	& {\bf 2} or {\bf 3} or {\bf 4} (s9)& $\ge 2$ (s7)	& $\ge 2$ (s7)			& $\ge 2$ (s7)	& ?     	\\ \hline
\bottomrule
\end{tabular}
\end{table}

A search for a $\PSCA(8,3,2)$ using left cosets in \cref{alg:subgroup} completed without output: 
if a $\PSCA(8,3,2)$ exists, then it does not occur as a union of distinct left cosets of a nontrivial subgroup of~$S_8$.
A search for a $\PSCA(8,4,2)$ and a $\PSCA(9,3,2)$ using left cosets in \cref{alg:subgroup} completed without output for a representative of each conjugacy class of subgroups of order greater than~$2$: if a $\PSCA(8,4,2)$ or $\PSCA(9,3,2)$ occurs as a union of distinct left cosets of a nontrivial subgroup of $S_8$ or~$S_9$, respectively, then the subgroup has order~$2$.

Although we obtained some positive results using \cref{alg:subgroup} for right cosets, the structure uncovered was less rich than for left cosets.
In particular, in each case tested we found that if a $\PSCA(n,k,\lambda)$ occurs as a union of distinct right cosets of a nontrivial subgroup $G$ of~$S_n$, then a $\PSCA(n,k,\lambda)$ also occurs as a union of distinct left cosets of a subgroup $G'$ of $S_n$, where $G'$ is isomorphic to~$G$. 
There were also several instances when the largest $|G|$ was smaller than the largest~$|G'|$, in which case the search using right cosets was signficantly slower (see \cref{sec:alg-desc}).
We therefore did not attempt to carry out some of the larger searches for right cosets.

\subsection{New parameter sets $(n,k,\lambda)$ for a $\PSCA(n,k,\lambda)$}
\label{sec:newparams}

Charlie Colbourn (personal communication, Sept.\ 2021) posed the following question. 
\begin{question}
\label{question:charlie}
Does the existence of a $\PSCA(n,k,\lambda)$ imply the existence of a $\PSCA(n,k,\lambda+1)$?
\end{question}
\noindent
This prompts the following observation as a direct consequence of \cref{lem:union}.  
\begin{lemma}
\label{lem:gto2g-1}
Let $k \le n$ and let $g = g(n,k)$.
Suppose there exists a $\PSCA(n,k,\lambda)$ for each 
\[
\lambda \in \{g, g+1, \dots, 2g-1\}.
\]
Then there exists a $\PSCA(n,k,\lambda)$ if and only if $\lambda \ge g$,
and the answer to Question~$\ref{question:charlie}$ is yes for the parameter pair $(n,k)$.
\end{lemma}

\cref{alg:subgroup} finds the following examples of a $\PSCA(5,3,3)$, a $\PSCA(6,3,3)$, a $\PSCA(7,3,3)$, a $\PSCA(7,4,3)$, a $\PSCA(8,3,4)$, and a $\PSCA(8,3,5)$ as a union of left cosets. Each of these parameter sets is new. 
Combination of these examples with the results in \cref{tab:gvalues-updated} shows that the answer to \cref{question:charlie} is yes for each of the parameter pairs
\[
(n,k) \in \{ (5,3), (6,3), (7,3), (7,4), (8,3) \},
\]
regardless of whether $g(8,3)=2$ or $g(8,3)=3$, and application of \cref{lem:gto2g-1} gives \cref{cor:alllambda}.
We do not currently know of a parameter set $(n,k,\lambda)$ for which the answer to \cref{question:charlie} is no.

\begin{proposition}
\label{prop:x33}
\mbox{}
\begin{enumerate}[$(i)$]

\item
%

The following $18$ sequences form a $\PSCA(5,3,3)$ as a union of $9$ left cosets of the subgroup 
$G = \langle 43215 \rangle \cong C_2$ of~$S_5$:
\[
\begin{array}{c||*{2}{c}}
\text{{\rm left coset}}   & \multicolumn{2}{c}{{\rm sequences}}		  \\ \hline
 12345\,G &   12345 & 43215 \\
 12435\,G &   12435 & 43125 \\
 13452\,G &   13452 & 42153 \\
 15324\,G &   15324 & 45231 \\
 21543\,G &   21543 & 34512 \\
 23514\,G &   23514 & 32541 \\
 24513\,G &   24513 & 31542 \\
 51423\,G &   51423 & 54132 \\
 52314\,G &   52314 & 53241 
\end{array}
\]

\item
%

The following $18$ sequences form a $\PSCA(6,3,3)$ as a union of $3$ left cosets of the subgroup 
$G = \langle 634215, 456123 \rangle \cong S_3$ of~$S_6$:
\[
\begin{array}{c||*{6}{c}}
\text{{\rm left coset}}   & \multicolumn{6}{c}{{\rm sequences}}		  \\ \hline
 123456\,G &   123456 & 215634 & 361542 & 456123 & 542361 & 634215 \\
 134265\,G &   134265 & 256143 & 315624 & 461532 & 523416 & 642351 \\
 162435\,G &   162435 & 241653 & 326514 & 435162 & 514326 & 653241 
\end{array}
\]

\item
%

The following $18$ sequences form a $\PSCA(7,3,3)$ as a union of $9$ left cosets of the subgroup 
$G = \langle 3412765 \rangle \cong C_2$ of~$S_7$:
\[
\begin{array}{c||*{2}{c}}
\text{{\rm left coset}}   & \multicolumn{2}{c}{{\rm sequences}}		  \\ \hline
 1234567\,G &   1234567 & 3412765 \\
 1253764\,G &   1253764 & 3471562 \\
 1643572\,G &   1643572 & 3621754 \\
 5147362\,G &   5147362 & 7325164 \\
 5241673\,G &   5241673 & 7423651 \\
 5276341\,G &   5276341 & 7456123 \\
 5432617\,G &   5432617 & 7214635 \\
 6175324\,G &   6175324 & 6357142 \\
 6245371\,G &   6245371 & 6427153 
\end{array}
\]

\item
%

The following $72$ sequences form a $\PSCA(7,4,3)$ as a union of $18$ left cosets of the subgroup 
$G = \langle 1576342 \rangle \cong C_4$ of~$S_7$:
\[
\begin{array}{c||*{4}{c}}
\text{{\rm left coset}}   & \multicolumn{4}{c}{{\rm sequences}}		  \\ \hline
 1234567\,G &   1234567 & 1324765 & 1576342 & 1756243 \\
 1256437\,G &   1256437 & 1376425 & 1534672 & 1724653 \\
 2136754\,G &   2136754 & 3126574 & 5174236 & 7154326 \\
 2164573\,G &   2164573 & 3164752 & 5146327 & 7146235 \\
 2315476\,G &   2315476 & 3217456 & 5713624 & 7512634 \\
 2436751\,G &   2436751 & 3426571 & 5674231 & 7654321 \\
 2537614\,G &   2537614 & 3725614 & 5372416 & 7253416 \\
 2574613\,G &   2574613 & 3754612 & 5326417 & 7236415 \\
 2716354\,G &   2716354 & 3516274 & 5214736 & 7314526 \\
 2741536\,G &   2741536 & 3541726 & 5261374 & 7361254 \\
 2764351\,G &   2764351 & 3564271 & 5246731 & 7346521 \\
 4127356\,G &   4127356 & 4135276 & 6152734 & 6173524 \\
 4162753\,G &   4162753 & 4163572 & 6145237 & 6147325 \\
 4251673\,G &   4251673 & 4371652 & 6531427 & 6721435 \\
 4263175\,G &   4263175 & 4362157 & 6547123 & 6745132 \\
 4523716\,G &   4523716 & 4732516 & 6275314 & 6357214 \\
 4536172\,G &   4536172 & 4726153 & 6254137 & 6374125 \\
 4576312\,G &   4576312 & 4756213 & 6234517 & 6324715 
\end{array}
\]

\item
%

The following $24$ sequences form a $\PSCA(8,3,4)$ as the single left coset $12354678\,G$ of the subgroup 
$G = \langle 67142358,\, 46572381 \rangle \cong {\rm SL}(2,3)$ of~$S_8$:
\[
\begin{array}{*{8}{c}}
12354678 & 16543872 & 18435276 & 24851637 & 26518734 & 27185436 & 31746825 & 35674128 \\
38467521 & 41752683 & 43275186 & 46527381 & 51738264 & 52387461 & 54873162 & 63241857 \\
67124358 & 68412753 & 73268514 & 74826315 & 75682413 & 82316745 & 85631247 & 87163542 
\end{array}
\]

\item
%

The following $30$ sequences form a $\PSCA(8,3,5)$ as a union of $5$ left cosets of the subgroup 
$G = \langle 65872143,\, 45712836 \rangle \cong S_3$ of~$S_8$:
\[
\begin{array}{c||*{6}{c}}
\text{{\rm left coset}}   & \multicolumn{6}{c}{{\rm sequences}}		  \\ \hline
 12345678\,G &   12345678 & 35182764 & 45712836 & 65872143 & 72465381 & 82635417 \\
 12485736\,G &   12485736 & 35842617 & 45162378 & 65732481 & 72615843 & 82375164 \\
 17384625\,G &   17384625 & 36148752 & 43761852 & 64837152 & 78416325 & 81673425 \\
 17564823\,G &   17564823 & 36278451 & 43281657 & 64217358 & 78536124 & 81543726 \\
 21685347\,G &   21685347 & 27315468 & 28475631 & 53742186 & 54862713 & 56132874 
\end{array}
\]

\end{enumerate}
\end{proposition}

\begin{corollary}
\label{cor:alllambda}
\mbox{}
\begin{enumerate}[$(i)$]
\item
For each $(n,k) \in \{ (5,3), (6,3), (7,3), (7,4) \}$, there exists a $\PSCA(n,k,\lambda)$ if and only if $\lambda \ge 2$.
\item
There exists a $\PSCA(8,3,\lambda)$ if and only if $\lambda \ge g(8,3)$, and $g(8,3) \in \{2,3\}$.
\end{enumerate}
\end{corollary}

\subsection{Examples of a group-based $\PSCA(n,n-1,1)$}
\label{sec:levgroup}

\cref{alg:subgroup} finds the following examples of a $\PSCA(n,n-1,1)$ as a union of $\frac{(n-1)!}{|G|}$ left cosets of a nontrivial subgroup $G$ of~$S_n$, for each $n \in \{4,5,6,7\}$. This suggests the possibility of a group-based proof of \cref{thm:lev}, as an alternative to Levenshtein's coding-theoretic proof. 

\begin{proposition}
The following sets of $(n-1)!$ sequences (listed only by reference to cosets) form a $\PSCA(n,n-1,1)$ as a union of $\frac{(n-1)!}{|G|}$ left cosets of the subgroup $G$ of~$S_n$. 
\begin{enumerate}[$(i)$]

\item
%

$n=4$ and the subgroup $G = \langle 3412 \rangle \cong C_2$ of~$S_4$:
\[
\begin{array}{*{3}{c}}
1234\,G & 1432\,G & 2413\,G
\end{array}
\]

\item
%

$n=5$ and the subgroup $G = \langle 34125, 43215 \rangle \cong C_2 \times C_2$ of~$S_5$:	
\[
\begin{array}{*{6}{c}}
12345\,G & 13254\,G & 14253\,G & 15243\,G & 15342\,G & 51432\,G 
\end{array}
\]

\item
%

$n=6$ and the subgroup $G = \langle 125634, 346521, 345612 \rangle \cong S_4$ of~$S_6$:
\[
\begin{array}{*{5}{c}}
123456\,G & 132546\,G & 132645\,G & 135642\,G & 136524\,G 
\end{array}
\]

\item

$n=7$ and the subgroup $G = \langle 7235461, 1756432 \rangle \cong S_4$ of~$S_7$:
\[
\begin{array}{*{6}{c}}
1234657\,G & 1235476\,G & 1237456\,G & 1273564\,G & 1324576\,G & 1325467\,G \\ 
1326475\,G & 1326574\,G & 1342756\,G & 1345267\,G & 1352764\,G & 1356427\,G \\
1357246\,G & 3124756\,G & 3125746\,G & 3126745\,G & 3127654\,G & 3145627\,G \\
3154267\,G & 3154762\,G & 3156742\,G & 3412657\,G & 3415276\,G & 3421567\,G \\
3425176\,G & 3426715\,G & 3427651\,G & 3456172\,G & 3457126\,G & 3457621\,G
\end{array}
\]

\end{enumerate}
\end{proposition}

\section{Open problems}
\label{sec:open}
We have established new values and bounds for $g(n,k)$, as shown in \cref{tab:gvalues-updated}.
We have established the following new parameter sets for a $\PSCA(n,k,\lambda)$: 
$(5,3,3)$, $(6,3,3)$, $(7,3,3)$, $(7,4,3)$, $(8,3,4)$, and $(8,3,5)$ (see \cref{sec:newparams}).
We have given an example of a $\PSCA(n,n-1,1)$ as a union of left cosets of a nontrivial subgroup of~$S_n$ for $n \in \{4,5,6,7\}$ (see \cref{sec:levgroup}).
Examples of a $\PSCA(n,k,\lambda)$ formed as a union of distinct left cosets of a nontrivial subgroup appear to be widespread,
and prescribing this structure brings into reach various searches that would otherwise be intractable.

We propose several open problems arising from our results.
\begin{enumerate}[$(i)$]
\item
Determine further exact values and bounds for $g(n,k)$.

\item
Find a $\PSCA(n,k,\lambda)$ for new parameter sets $(n,k,\lambda)$.

\item
Find a group-based construction for a $\PSCA(n,n-1,1)$ for each $n\ge 2$.

\item
Is there a parameter set $(n,k,\lambda)$ for which a $\PSCA(n,k,\lambda)$ exists but there is no example that is a union of left cosets of a nontrivial subgroup of~$S_n$?

\item
Resolve \cref{question:charlie} by determining whether the existence of a $\PSCA(n,k,\lambda)$ implies the existence of a $\PSCA(n,k,\lambda+1)$.

\item 
The recursive search methods presented here appear to encounter memory constraints when attempting to settle the smallest open case of \cref{conj:mvtequiv}, namely whether $g(9,7) > 1$.
 Are there theoretical techniques or improved search methods for handling this case?

\item
Find more combinatorial nonexistence results for perfect sequence covering arrays.
\end{enumerate}

\section*{Comments}
The authors thank Dan Gentle and Ian Wanless for kindly sharing a preprint of the paper~\cite{gentle-wanless:psca}, which describes how they used different methods from ours to computationally determine existence and nonexistence results for perfect sequence covering arrays that complement ours.
They independently showed that $g(6,3)=g(7,3)=2$.
They also recovered some of the results originally reported in~\cite{na-masters}, in particular that $g(7,4)=2$. 
They further established two results that our methods were not able to find: $g(8,3)>2$ and $g(8,4)>2$, the first of which combines with the $\PSCA(8,3,3)$ of \cref{prop:newg-subgroup} to show that $g(8,3)=3$.
Conversely, our methods found results that were not obtained in \cite{gentle-wanless:psca}, including that $g(7,5) \le 4$ and $g(9,3) \le 4$.
The paper \cite{gentle-wanless:psca} also gives various bounds on the value of $g(n,k)$ that arise from examples of perfect sequence covering arrays comprising a complete subgroup of~$S_n$.

The authors gratefully acknowledge helpful discussions with Karen Meagher and Charlie Colbourn.


\end{document}